\documentclass[12pt]{amsart}

\usepackage{amssymb}
\usepackage{amscd}

\title{Algebraic curves with collinear Galois points} 
\author{Satoru Fukasawa}

\subjclass[2010]{14H50, 14H05, 14H37}
\keywords{Galois point, plane curve, Galois group, automorphism group}
\address{Department of Mathematical Sciences, Faculty of Science, Yamagata University, Kojirakawa-machi 1-4-12, Yamagata 990-8560, Japan} 
\email{s.fukasawa@sci.kj.yamagata-u.ac.jp}

\thanks{The author was partially supported by JSPS KAKENHI Grant Number JP19K03438.}

\newtheorem{theorem}{Theorem}
\newtheorem{proposition}{Proposition}
\newtheorem{corollary}{Corollary} 
\newtheorem{fact}{Fact}

\theoremstyle{definition}

\begin{document}
\begin{abstract} 
A criterion for the existence of a birational embedding into a projective plane with three collinear Galois points for algebraic curves is presented. 
The extendability of an automorphism induced by a Galois point to a linear transformation of the projective plane is also discussed, under the assumption that two Galois points exist. 
\end{abstract}

\maketitle 

\section{Introduction} 
Let $X$ be a (reduced, irreducible) smooth projective curve over an algebraically closed field $k$ of characteristic $p \ge 0$ and let $k(X)$ be its function field. 
We consider a morphism $\varphi: X \rightarrow \mathbb{P}^2$, which is birational onto its image.
In this situation, Hisao Yoshihara introduced the notion of a Galois point.  
A point $P \in \mathbb{P}^2$ is called a {\it Galois point}, if the field extension $k(\varphi(X))/\pi_P^*k(\mathbb{P}^1)$ of function fields induced by the projection $\pi_P$ from $P$ is a Galois extension (\cite{miura-yoshihara, yoshihara}).  
Furthermore, a Galois point $P$ is said to be inner (resp. outer), if $P \in \varphi(X) \setminus {\rm Sing}(\varphi(X))$ (resp. if $P \in \mathbb{P}^2 \setminus \varphi(X)$). 

A criterion for the existence of a birational embedding with two Galois points was described by the present author (\cite{fukasawa2}). 
It is a natural problem to find a condition for the existence of {\it three} Galois points (see also \cite{open}). 
Non-collinear Galois points were considered in \cite{fukasawa3}. 
In this article, (three) collinear Galois points are studied. 
The associated Galois group is denoted by $G_P$, when $P$ is a Galois point.
The following criterion is presented. 

\begin{theorem} \label{main} 
Let $G_1$, $G_2$ and $G_3 \subset {\rm Aut}(X)$ be finite subgroups of order at least three, and let $P_1$, $P_2$ and $P_3$ be different points of $X$. 
Then, four conditions
\begin{itemize}
\item[(a)] $X/{G_i} \cong \Bbb P^1$ for $i=1, 2, 3$,    
\item[(b)] $G_i \cap G_j=\{1\}$ for any $i, j$ with $i \ne j$,  
\item[(c)] there exists a divisor $D$ such that $D=P_i+\sum_{\sigma \in G_i}\sigma(P_j)$ for any $i, j$ with $i \ne j$, and
\item[(d)] $\dim \Lambda \le 2$, for the smallest sublinear system $\Lambda$ of $|D|$ such that $D, P_i+\sum_{\sigma \in G_i}\sigma(P_i) \in \Lambda$ for $i=1, 2, 3$
\end{itemize}
are satisfied, if and only if there exists a birational embedding $\varphi: X \rightarrow \mathbb P^2$ of degree $|G_1|+1$ such that $\varphi(P_1)$, $\varphi(P_2)$ and $\varphi(P_3)$ are three collinear inner Galois points for $\varphi(X)$ and $G_{\varphi(P_i)}=G_i$ for $i=1, 2, 3$. 
\end{theorem}

\begin{theorem} \label{main-outer}
Let $G_1$, $G_2$ and $G_3 \subset {\rm Aut}(X)$ be finite subgroups, and let $Q$ be a point of $X$. 
Then, four conditions
\begin{itemize}
\item[(a)] $X/{G_i} \cong \Bbb P^1$ for $i=1, 2, 3$,    
\item[(b)] $G_i \cap G_j=\{1\}$ for any $i, j$ with $i \ne j$,  
\item[(c')] there exists a divisor $D$ such that $D=\sum_{\sigma \in G_i}\sigma(Q)$ for $i=1, 2, 3$, and
\item[(d')] $\dim \Lambda \le 2$, for the smallest sublinear system $\Lambda$ of $|D|$ such that $\Lambda_1 \cup \Lambda_2 \cup \Lambda_3 \subset \Lambda$, where $\Lambda_i$ is the base-point-free linear system induced by the covering map $X \rightarrow X/G_i \cong \mathbb{P}^1$ for $i=1, 2, 3$
\end{itemize}
are satisfied, if and only if there exists a birational embedding $\varphi: X \rightarrow \mathbb P^2$ of degree $|G_1|$ and three collinear outer Galois points $P_1, P_2$ and $P_3$ exist for $\varphi(X)$ such that $G_{P_i}=G_i$ for $i=1, 2, 3$, and $\overline{P_1P_2} \ni \varphi(Q)$, where $\overline{P_1P_2}$ is the line passing through $P_1$ and $P_2$. 
\end{theorem}

The uniqueness of the birational embedding constructed in \cite{fukasawa2} is also proved. 

\begin{proposition} \label{uniqueness} 
Assume that the orders of groups $G_1$ and $G_2$ in Facts \ref{criterion1} and \ref{criterion2} are at least three. 
Then, a morphism $\varphi$ described in Fact \ref{criterion1} (resp. in Fact \ref{criterion2}) is uniquely determined by a $4$-tuple $(G_1, G_2, P_1, P_2)$ (resp. by a $3$-tupble $(G_1, G_2, Q)$), up to a projective equivalence. 
\end{proposition} 

Using (the proof of) this Proposition, the following criterion for the extendability of an automorphism $\sigma \in G_P$ for an inner Galois point $P$ is presented. 

\begin{proposition} \label{extendable} 
Let $\deg \varphi(X) \ge 4$, let $\varphi(P_1)$ and $\varphi(P_2) \in \varphi(X) \subset \mathbb{P}^2$ be different inner Galois points, and let $\sigma \in G_{\varphi(P_1)}$ satisfy $P_3=\sigma(P_2)$. 
Then, there exists a linear transformation $\tilde{\sigma}$ of $\mathbb{P}^2$ such that $\varphi^{-1}\circ\tilde{\sigma}\circ\varphi=\sigma$, if and only if three conditions
\begin{itemize}
\item[(a)] $\sigma(P_1)=P_1$, 
\item[(b)] $\varphi(P_3)$ is an inner Galois point, and 
\item[(c)] $\sigma^*(P_3+\sum_{\gamma \in G_{\varphi(P_3)}}\gamma(P_3))=P_2+\sum_{\tau \in G_{\varphi(P_2)}}\tau(P_2)$
\end{itemize}
are satisfied. 
\end{proposition}

\begin{corollary} \label{total flexes}
Let $\varphi(P_1), \varphi(P_2)$ and $\varphi(P_3)$ be different inner Galois points, and let $\sigma \in G_{\varphi(P_1)}$ satisfy $\sigma(P_2)=P_3$. 
If $\varphi(P_1)$, $\varphi(P_2)$ and $\varphi(P_3)$ are total inflection points, then there exists a linear transformation $\tilde{\sigma}$ of $\mathbb{P}^2$ such that $\varphi^{-1}\circ\tilde{\sigma}\circ\varphi=\sigma$. 
\end{corollary}

\section{Preliminaries} 

We recall the criterion presented in \cite{fukasawa2} for two Galois points. 

\begin{fact} \label{criterion1} 
Let $G_1$ and $G_2$ be finite subgroups of ${\rm Aut}(X)$ and let $P_1$ and $P_2$ be different points of $X$.
Then, three conditions
\begin{itemize}
\item[(a)] $X/{G_i} \cong \Bbb P^1$ for $i=1, 2$,    
\item[(b)] $G_1 \cap G_2=\{1\}$, and
\item[(c)] $P_1+\sum_{\sigma \in G_1} \sigma (P_2)=P_2+\sum_{\tau \in G_2} \tau (P_1)  $
\end{itemize}
are satisfied, if and only if there exists a birational embedding $\varphi: X \rightarrow \mathbb P^2$ of degree $|G_1|+1$ such that $\varphi(P_1)$ and $\varphi(P_2)$ are different inner Galois points for $\varphi(X)$ and $G_{\varphi(P_i)}=G_i$ for $i=1, 2$. 
\end{fact}

\begin{fact} \label{criterion2} 
Let $G_1$ and $G_2$ be finite subgroups of ${\rm Aut}(X)$ and let $Q$ be a point of $X$.
Then, three conditions
\begin{itemize}
\item[(a)] $X/{G_i} \cong \Bbb P^1$ for $i=1, 2$,    
\item[(b)] $G_1 \cap G_2=\{1\}$, and
\item[(c')] $\sum_{\sigma \in G_1} \sigma (Q)=\sum_{\tau \in G_2} \tau (Q)  $
\end{itemize}
are satisfied, if and only if there exists a birational embedding $\varphi: X \rightarrow \mathbb P^2$ of degree $|G_1|$ and two outer Galois points $P_1$ and $P_2$ exist for $\varphi(X)$ such that $G_{\varphi(P_i)}=G_i$ for $i=1, 2$, and $\overline{P_1P_2} \ni Q$. 
\end{fact}

According to \cite[Lemma 2.5]{fukasawa1}, the following holds. 

\begin{fact} \label{inner-lemma}
Assume that $\deg \varphi(X) \ge 4$, and points $\varphi(P_1)$ and $\varphi(P_2)$ are distinct inner Galois points for $\varphi(X)$. 
Then, the line $\overline{\varphi(P_1)\varphi(P_2)}$ is different from the tangent line at $\varphi(P_1)$. 
In particular, $\sigma(P_1) \ne P_2$ for each automorphism $\sigma \in G_{\varphi(P_1)}$. 
\end{fact}
\section{Proof of Theorems \ref{main} and \ref{main-outer}}

\begin{proof}[Proof of Theorem \ref{main}]
We consider the if-part. 
It follows from conditions (a) and (b) in Fact \ref{criterion1} that conditions (a) and (b) are satisfied. 
By Fact \ref{criterion1}(c), since $\varphi(P_1), \varphi(P_2)$ and $\varphi(P_3)$ are collinear Galois points, condition (c) is satisfied. 
Let $\Lambda' \subset |D|$ be the (base-point-free) linear system induced by $\varphi$. 
Since $\varphi(P_i)$ is inner Galois, $P_i+\sum_{\sigma \in G_i}\sigma(P_i) \in \Lambda'$, for $i=1, 2, 3$. 
Therefore, $\dim \Lambda \le 2$. 
Condition (d) is satisfied. 

We consider the only-if part. 
By conditions (a), (b) and (c) and Fact \ref{criterion1}, for each $i, j$ with $i \ne j$, there exists a birational embedding $\varphi_{ij}: X \rightarrow \mathbb{P}^2$ such that $\varphi_{ij}(P_i)$ and $\varphi_{ij}(P_j)$ are inner Galois points for $\varphi_{ij}(X)$, $G_{\varphi_{ij}(P_i)}=G_i$ and $G_{\varphi_{ij}(P_j)}=G_j$. 
It follows from Fact \ref{inner-lemma} that 
$$ G_1P_1 \ne G_1 P_2, \mbox{ and } \sum_{\sigma \in G_1}\sigma(P_1) \ne \sum_{\sigma \in G_1} \sigma(P_2). $$
Then, by condition (a), there exists a function $f \in k(X) \setminus k$ such that 
$$ k(X)^{G_1}=k(f), \ (f)=\sum_{\sigma \in G_1}\sigma(P_1)-\sum_{\sigma \in G_1}\sigma(P_2)$$
(see also \cite[III.7.1, III.7.2, III.8.2]{stichtenoth}). 
Note that, by condition (c), $(f)_{\infty}=D-P_1$. 
Similarly, there exist $g, h \in k(X) \setminus k$ such that 
$$ k(X)^{G_2}=k(g), \ (g)=\sum_{\tau \in G_2}\tau(P_2)-(D-P_2) $$
and 
$$ k(X)^{G_3}=k(h), \ (h)=\sum_{\gamma \in G_3}\gamma(P_3)-(D-P_3).$$ 
Then, $\varphi_{12}$ is represented by $(f:g:1)$ (see \cite[Proofs of Proposition 1 and of Theorem 1]{fukasawa2}). 
Let $\Lambda \subset |D|$ be as in condition (d), and let $\Lambda' \subset |D|$ be the sublinear system corresponding to $\langle f, g, 1 \rangle$.
Since $D, (f)+D, (g)+D \in \Lambda$, it follows that $\Lambda' \subset \Lambda$.  
By condition (d), $\Lambda'=\Lambda$. 
This implies that $P_3+\sum_{\gamma \in G_3}\gamma(P_3) \in \Lambda'$. 
Therefore, $h \in \langle f, g, 1\rangle$. 
Since the covering map $X \rightarrow X/G_3$ is represented by $\langle h, 1 \rangle$, this covering map coincides with the projection from some smooth point of $\varphi_{12}(X)$. 
Such a center of projection coincides with $\varphi_{12}(P_3)$, since the center is determined by ${\rm supp}(D) \cap {\rm supp}((h)+D)$. 
This implies that $\varphi_{12}(P_3)$ is an inner Galois point. 
By condition (c), points $\varphi_{12}(P_1)$, $\varphi_{12}(P_2)$ and $\varphi_{12}(P_3)$ are collinear. 
\end{proof}

\begin{proof}[Proof of Theorem \ref{main-outer}]
We consider the if-part. 
It follows from conditions (a) and (b) in Fact \ref{criterion2} that conditions (a) and (b) are satisfied. 
By Fact \ref{criterion2}(c'), since $P_1, P_2$ and $P_3$ are collinear outer Galois points, condition (c') is satisfied. 
Let $\Lambda' \subset |D|$ be the (base-point-free) linear system induced by $\varphi$. 
Since $P_i$ is outer Galois, the linear system corresponding to $X \rightarrow X/G_i \cong \mathbb{P}^1$ is contained in $\Lambda'$, for $i=1, 2, 3$. 
Therefore, $\dim \Lambda \le 2$. 
Condition (d') is satisfied. 

We consider the only-if part. 
By condition (a), there exists a function $f \in k(X) \setminus k$ such that 
$$ k(X)^{G_1}=k(f), \ (f)_{\infty}=\sum_{\sigma \in G_1}\sigma(Q)$$
(see also \cite[III.7.1, III.7.2, III.8.2]{stichtenoth}). 
Note that, by condition (c'), $(f)_{\infty}=D$. 
The sublinear system corresponding to $\langle 1, f \rangle \subset \mathcal{L}(D)$ coincides with $\Lambda_1 \subset |D|$ as in condition (d'). 
Similarly, there exist $g, h \in k(X) \setminus k$ such that 
$$ k(X)^{G_2}=k(g), \ k(X)^{G_3}=k(h), \mbox{ and } \ (g)_{\infty}=(h)_{\infty}=D. $$
Furthermore, the subspaces $\langle 1, g \rangle$ and $\langle 1, h \rangle$ correspond to the linear systems $\Lambda_2$ and $\Lambda_3$ as in condition (d'), respectively. Then, by conditions (b) and (c'), the morphism $\varphi$ represented by $(f:g:1)$ is birational onto its image and outer Galois points $P_1$ and $P_2$ exist for $\varphi(X)$ such that $G_{\varphi(P_i)}=G_i$ for $i=1, 2$ (see \cite[Proofs of Proposition 1 and of Theorem 1]{fukasawa2}). 
Let $\Lambda \subset |D|$ be as in condition (d'), and let $\Lambda' \subset |D|$ be the sublinear system corresponding to $\langle f, g, 1 \rangle$. 
Since $\Lambda_1, \Lambda_2 \subset \Lambda$, it follows that $\Lambda' \subset \Lambda$.  
By condition (d'), $\Lambda'=\Lambda$. 
This implies that $\Lambda_3 \subset \Lambda'$. 
Therefore, $h \in \langle f, g, 1\rangle$. 
Since the covering map $X \rightarrow X/G_3$ is represented by $\langle h, 1 \rangle$, this covering map coincides with the projection from some outer point $P_3 \in \mathbb{P}^2 \setminus \varphi(X)$. 
Then, $P_3$ is an outer Galois point. 
By condition (c'), points $P_1, P_2$ and $P_3$ are collinear. 
\end{proof}

\section{Proof of Propositions \ref{uniqueness} and \ref{extendable}}

\begin{proof}[Proof of Proposition \ref{uniqueness}]
We consider inner Galois points. 
Assume that condition (c) in Fact \ref{criterion1} is satisfied. 
Let $D:=P_1+\sum_{\sigma \in G_1}\sigma(P_2)=P_2+\sum_{\tau \in G_2}\tau(P_1)$. 
Note that, by Fact \ref{inner-lemma}, $P_1+\sum_{\sigma \in G_1}\sigma(P_1) \ne D$ and $P_1 \not\in {\rm supp}(P_2+\sum_{\tau \in G_2}\tau(P_2))$. 
The uniqueness of the linear system corresponding to a birational embedding follows, since a (base-point-free) linear system $\Lambda \subset |D|$ of dimension two such that 
$$D, \ P_1+\sum_{\sigma \in G_1}\sigma(P_1), \ P_2+\sum_{\tau \in G_2}\tau(P_2) \in \Lambda$$
is uniquely determined. 

We consider outer Galois points. 
Assume that condition (c') in Fact \ref{criterion2} is satisfied. 
Let $D:=\sum_{\sigma \in G_1}\sigma(Q)=\sum_{\tau \in G_2}\tau(Q)$, and let $\Lambda_i$ be the (base-point-free) linear system corresponding to the covering map $\pi_i: X \rightarrow X/G_i \cong \mathbb{P}^1$ for $i=1, 2$. 
Then, $D \in \Lambda_i$ and $\Lambda_i \subset |D|$. 
If $\pi_1$ and $\pi_2$ are realized as the projections from different outer Galois points for a birational embedding $\varphi: X \rightarrow \mathbb{P}^2$, then $\varphi$ is determined by a sublinear system $\Lambda \subset |D|$ such that $\dim \Lambda=2$ and $\Lambda_1 \cup \Lambda_2 \subset \Lambda$, up to a projective equivalence.  
Therefore, the uniqueness follows. 
\end{proof}

\begin{proof}[Proof of Proposition \ref{extendable}] 
Let $C:=\varphi(X)$. 
We consider the only-if part. 
Assume that there exists a linear transformation $\tilde{\sigma}$ of $\mathbb{P}^2$ such that $\varphi^{-1}\circ\tilde{\sigma}\circ\varphi=\sigma$.  
For a general line $\ell \ni \varphi(P_1)$, $C \cap \ell$ contains at least two points (since $\deg C \ge 3$), and $\tilde{\sigma}((C \cap \ell)\setminus \{\varphi(P_1)\}) \subset \ell$. 
Since $\tilde{\sigma}$ is a linear transformation, $\tilde{\sigma}(\ell)=\ell$ follows. 
This implies that $\tilde{\sigma}(\varphi(P_1))=\varphi(P_1)$. 
Condition (a) is satisfied. 
Since $\varphi(P_3)=\varphi(\sigma(P_2))=\tilde{\sigma}(\varphi(P_2))$, condition (b) is satisfied. 
Since the divisor $P_2+\sum_{\tau \in G_{\varphi(P_2)}}\tau(P_2)$ corresponds to the tangent line of $\varphi(X)$ at $\varphi(P_2)$, conditions (c) is also satisfied. 

We consider the if part. 
Let $\Lambda$ be the linear system corresponding to the birational embedding $\varphi: X \rightarrow \mathbb{P}^2$. 
As in the proof of Proposition \ref{uniqueness}, it follows from condition (b) that $\Lambda$ is the smallest linear system containing the divisors 
$$D, \ P_1+\sum_{\sigma \in G_{\varphi(P_1)}}\sigma(P_1), \ P_3+\sum_{\gamma \in G_{\varphi(P_3)}}\gamma(P_3), $$ 
where $D:=P_1+\sum_{\sigma \in G_{\varphi(P_1)}}\sigma(P_3)=P_3+\sum_{\gamma \in G_{\varphi(P_3)}}\gamma(P_1)$. 
By condition (a), divisors $D$ and $P_1+\sum_{\sigma \in G_{\varphi(P_1)}}\sigma(P_1)$ are invariant under the action of $\sigma^*$. 
Since $P_2+\sum_{\tau \in G_{\varphi(P_2)}}\tau(P_2) \in \Lambda$, by condition (c), it follows that $\sigma^*\Lambda=\Lambda$. 
\end{proof} 

\begin{proof}[Proof of Corollary \ref{total flexes}] 
We prove that conditions (a), (b) and (c) in Proposition \ref{extendable} are satisfied. 
Since $\varphi(P_1)$ is a total inflection point, by \cite[III.8.2]{stichtenoth}, condition (a) is satisfied.  
Condition (b) is satisfied by the assumption. 
Since $\varphi(P_3)$ is a total inflection point, it follows from \cite[III.8.2]{stichtenoth} that  
$$ P_3+\sum_{\gamma \in G_{\varphi(P_3)}}\gamma(P_3)=(|G_3|+1) P_3. $$
Therefore, 
$$ \sigma^*\left(P_3+\sum_{\gamma \in G_{\varphi(P_3)}}\gamma(P_3)\right)=(|G_2|+1)P_2= P_2+\sum_{\tau \in G_{\varphi(P_2)}}\tau(P_2). $$
Condition (c) is satisfied.  
\end{proof}

\begin{center} {\bf Acknowledgements} \end{center} 
The author is grateful to Doctor Kazuki Higashine for helpful discussions.


\begin{thebibliography}{20} 
\bibitem{fukasawa1} S. Fukasawa, An upper bound for the number of Galois points for a plane curve, Topics in Finite Fields, pp.111--119, Contemp. Math. {\bf 632}, Amer. Math. Soc., Providence, RI, 2015. 
\bibitem{fukasawa2} S. Fukasawa, A birational embedding of an algebraic curve into a projective plane with two Galois points, J. Algebra {\bf 511} (2018), 95--101.
\bibitem{fukasawa3} S. Fukasawa, Algebraic curves admitting non-collinear Galois points, preprint, arXiv:1908.00259. 
\bibitem{miura-yoshihara} K. Miura and H. Yoshihara, Field theory for function fields of plane quartic curves, J. Algebra {\bf 226} (2000), 283--294. 
\bibitem{stichtenoth} H. Stichtenoth, {\it Algebraic function fields and codes}, Universitext, Springer-Verlag, Berlin, 1993. 
\bibitem{yoshihara} H. Yoshihara, Function field theory of plane curves by dual curves, J. Algebra {\bf 239} (2001), 340--355. 
\bibitem{open} H. Yoshihara and S. Fukasawa, List of problems, available at: \\ http://hyoshihara.web.fc2.com/openquestion.html
\end{thebibliography}
\end{document}